\documentclass[a4paper,12pt]{article}

\usepackage{mathrsfs}
\usepackage{stmaryrd}
\usepackage{amsmath}
\usepackage{amssymb}
\usepackage{amsfonts}
\usepackage{amsthm}
\usepackage{epsfig}
\usepackage{graphics,graphicx,color}
\usepackage{textcomp}
\usepackage{hyperref}
\usepackage{setspace}
\usepackage{soul}
\usepackage[authoryear]{natbib}

\usepackage{enumerate}

\theoremstyle{definition} \newtheorem{lemma}{Lemma}
\theoremstyle{definition} \newtheorem{theorem}{Theorem}
\theoremstyle{definition} 
\theoremstyle{definition} 
\theoremstyle{definition}

\title{Prediction based on the Kennedy-O'Hagan calibration model: asymptotic consistency and other properties}

\author{Rui Tuo \\ Academy of Mathematics and Systems Sciences\\ Chinese Academy of Sciences \and C. F. Jeff Wu\\ School of Industrial and Systems Engineering\\ Georgia Institute of Technology}

\begin{document}

\maketitle

\abstract{\cite{kennedy2001bayesian} propose a model for calibrating some unknown parameters in a computer model and estimating the discrepancy between the computer output and physical response. This model is known to have certain identifiability issues. \cite{tuo2014calibration} show that there are examples for which the Kennedy-O'Hagan method renders unreasonable results in calibration. In spite of its unstable performance in calibration, the Kennedy-O'Hagan approach has a more robust behavior in predicting the physical response. In this work, we present some theoretical analysis to show the consistency of predictor based on their calibration model in the context of radial basis functions.}

\vspace{9pt}

\section{Introduction}

With the development of mathematical models and computational technique, simulation programs or softwares are shown to be increasingly powerful for the prediction, validation and control of many physical processes. A computer simulation run, based on a virtual platform, requires only computational resources that are rather inexpensive in today's computing environment. In contrast, a physical experiment usually requires more facilities, materials, and human labor. As a consequence, a typical computer simulation run is much cheaper than its corresponding physical experiment trial. The economic benefits of computer simulations make them particularly useful and attractive in scientific and engineering research.
As a branch of statistics, \textit{design of experiments} mainly studies the methodologies on the planning, analysis and optimization of physical experiments \citep{wu2011experiments}. Given the rapid spread of computer simulations, it is beneficial to develop theory and methods for the design and analysis of computer simulation experiments. This emerging field is commonly referred to as \textit{computer experiments}. We refer to \cite{santner2003design} for more details.

The input variables of a computer experiment normally consist of factors which can be controlled in the physical process, referred to as the \textit{control variables}, as well as some model parameters. These model parameters represent certain intrinsic properties of the physical system. For example, to simulate a heat transfer process, we need to solve a heat equation. The formulation of the equation requires the environmental settings and the initial conditions of the system which can be controlled physically, as well as the thermal conductivity which is uncontrollable and cannot be measured directly in general. For most computer simulations, the prediction accuracy of the computer model is closely related to the choice of the model parameters. A standard method for determining the unknown model parameters is to estimate them by comparing the computer outputs and the physical responses. Such a procedure is known as the \textit{calibration} for computer models, and the model parameters to be identified are called the \textit{calibration parameters}. \cite{kennedy2001bayesian} first study the calibration problem using ideas and methods in computer experiments. They propose a Bayesian hierarchical model to estimate the calibration parameters by computing their posterior distributions. \cite{tuo2014calibration} show that the Kennedy-O'Hagan method may render unreasonable estimates for the calibration parameters. Given the widespread use of the Kennedy-O'Hagan method, it will be desirable to make a comprehensive assessment about this method. For brevity, we sometimes \textit{refer to Kennedy-O'Hagan as KO}.

This paper endeavors to study the prediction performance of the Kennedy-O'Hagan approach. First we adopt the framework of \cite{tuo2014calibration} which assumes the physical observations to be non-random. Interpolation theory in the native spaces becomes the key mathematical tool in this part. Then, we study the more realistic situation where the physical data are noisy. We employ the asymptotic theory of the smoothing splines in the Sobolev spaces to obtain the rate of convergence of the KO predictor in this case.

This article is organized as follows. In Section \ref{Sec:KO} we review the Bayesian method proposed by \cite{kennedy2001bayesian} for calibrating the model parameters and predicting for new physical responses. In Section \ref{Sec:Theory} we present our main results on the asymptotic theory on the prediction performance of the KO method. Concluding remarks and further discussions are made in Section \ref{Sec:discussions}. Some technical proofs are given in Appendix \ref{App:proof}.

\section{Review on the Kennedy-O'Hagan Method}\label{Sec:KO}

In this section we review the Bayesian method proposed by \cite{kennedy2001bayesian}. The formulation of this approach can be generalized to some extend. See, for example, \cite{higdon2004combining}.

Denote the experimental region for the control variables as $\Omega$. We suppose that $\Omega$ is a convex and compact subset of $\mathbb{R}^d$. Let $\{x_1,\ldots,x_n\}\subset\Omega$ be the set of design points for the physical experiment. Denote the responses of the $n$ physical experimental runs by $y_1^p,\ldots,y_n^p$ respectively, with $p$ standing for ``physical''. Let $\Theta$ be the domain of the calibration parameter. In this article, we suppose the computer model is \textit{deterministic}, i.e., the computer output is a deterministic function of the control variables and the calibration parameters, denoted by $y^s(x,\theta)$ for $x\in\Omega,\theta\in\Theta$ with $s$ standing for ``simulation''.

We consider two types of computer models. The first is called ``cheap computer simulations''. In these problems each run of the computer code takes only a short time so that we can call the computer simulation code inside our statistical analysis program which is usually based on an iterative algorithm like the Markov Chain Monte Carlo (MCMC). The second is called ``expensive computer simulations''. In these problems each run of the computer code takes a long time so that it is unrealistic to embed the computer simulation code into an iterative algorithm. A standard approach in computer experiments is to run the computer code over a set of selected points, and build a \textit{surrogate model} based on the obtained computer outputs to approximate the underlying true function. The surrogate model can be evaluated much faster. In the statistical analysis, the response values from the surrogate model are used instead of those from the original computer model.

\subsection{The Case of Cheap Computer Simulations}\label{Sec:cheap}

We model the physical response $y^p$ in the following nonparametric manner
\begin{eqnarray}
y_i^p=\zeta(x_i)+e_i,\label{NP}
\end{eqnarray}
where $\zeta(\cdot)$ is an underlying function, referred to as the \textit{true process}, and $e_i$'s are the observation error. We assume $e_i$'s are independent and identically distributed normal random variables with mean zero and unknown variance $\sigma^2$.
The computer output function and the physical true process are linked by
\begin{eqnarray}
\zeta(\cdot)=y^s(\cdot,\theta_0)+\delta(\cdot),\label{KO}
\end{eqnarray}
where $\theta_0$ denotes the ``true'' calibration parameter (from a physical point of view), and $\delta$ denotes an underlying discrepancy function between the physical process and the computer model under the true calibration parameters. It is reasonable to believe that in most computer experiment problems, the discrepancy function $\delta$ should be nonzero and possibly highly nonlinear because the computer codes are usually built under assumptions or simplifications that do not hold true in reality.

To estimate $\theta_0$ and $\delta$, we follow a standard Bayesian procedure by imposing certain prior distributions on the unknown parameters $\theta_0$ and $\sigma^2$ and the unknown function $\delta(\cdot)$. In the computer experiment literature, a prominent method is to use a Gaussian process as the prior for an unknown function \citep{santner2003design}. There are two major reasons for choosing Gaussian processes. First, the sample paths of a Gaussian process are smooth if a smooth covariance function is chosen, which can be beneficial when the target function is smooth as well. Second, the computational burden of the statistical inference and prediction for a Gaussian process model is relatively low. Specifically, we use a Gaussian process with mean zero and covariance function $\tau^2 C_\gamma(\cdot,\cdot)$ as the prior of $\delta(\cdot)$, where $C_\gamma$ is a stationary kernel with hyper-parameter $\gamma$.

In view of the finite-dimensional distribution of a Gaussian process, given $\tau^2$ and $\gamma$, $\delta(\mathbf{x})=(\delta(x_1),\ldots,\delta(x_n))^\text{T}$ follows the multivariate normal distribution $N(0,\tau^2\Sigma_\gamma)$, where $\Sigma_\gamma=(C_\gamma(x_i,x_j))_{i j}$. In order to discuss the prediction problem later, we apply the data augmentation algorithm of \cite{tanner1987calculation} and consider the posterior distribution of $(\theta_0,\delta(\mathbf{x}),\sigma^2,\gamma)$ given by
\begin{eqnarray}
&&\pi(\theta_0,\delta(\mathbf{x}),\tau^2,\sigma^2,\gamma|\mathbf{y}^p)\nonumber\\
&\propto& \pi(\mathbf{y}^p|\theta_0,\delta(\mathbf{x}),\tau^2,\sigma^2,\gamma) \pi(\delta(\mathbf{x})|\theta_0,\tau^2,\sigma^2,\gamma)\pi(\theta_0,\tau^2,\sigma^2,\gamma)\nonumber
\\&\propto& \sigma^{-n/2}\exp\left\{-\frac{1}{2\sigma^2}\|\mathbf{y}^p-y^s(\mathbf{x},\theta_0)-\delta(\mathbf{x})\|^2\right\}\nonumber\\
&&\times\tau^{-n/2}\left(\det\Sigma_\gamma\right)^{-1/2} \exp\left\{-\frac{\delta(\mathbf{x})^\text{T}\Sigma_\gamma^{-1}\delta(\mathbf{x})}{2\tau^2}\right\}\pi(\theta_0,\tau^2, \sigma^2, \gamma),\label{cheap}
\end{eqnarray}
where $\mathbf{y}^p=(y^p_1,\ldots,y^p_n)^\text{T}, y^s(\mathbf{x},\theta_0)=(y^s(x_1,\theta_0),\ldots,y^s(x_n,\theta_0))^\text{T}$. 
It is not time-consuming to evaluate the posterior density function $\pi(\cdot,\cdot,\cdot,\cdot,\cdot|\mathbf{y}^p)$ because the computer code is cheap to run.
A standard MCMC procedure can then be employed to draw samples from the posterior distribution. We refer to \cite{higdon2004combining} for further details.

In this work, we pay special attention to the prediction for a new physical reponse at an untried point $x_{new}$, denoted as $y^p(x_{new})$. Samples from the posterior predictive distribution of $y^p(x_{new})$ can be drawn along with the MCMC sampling. To see this, we note that in view of the Gaussian process assumption, given $\delta(\mathbf{x})$ and $\gamma$, $\delta(x_{new})$ follows the normal distribution
\begin{eqnarray*}
N(\Sigma_1^\text{T}\Sigma_\gamma^{-1}\delta(\mathbf{x}),\tau^2(C_\gamma(x_{new},x_{new}) -\Sigma_1^\text{T}\Sigma_\gamma^{-1}\Sigma_1)),
\end{eqnarray*}
where $\Sigma_1=(C_\gamma(x_1,x_{new}),\ldots,C_\gamma(x_n,x_{new}))^\text{T}$. Because in each iteration of the MCMC procedure a sample of $(\delta(\mathbf{x}),\theta_0,\gamma,\sigma^2)$ is drawn, we can draw a sample of $y^p(x_{new})$ from its posterior distribution $\pi(y^p(x_{new})|\mathbf{y}^p,\delta(\mathbf{x}),\theta_0,\tau^2,\gamma,\sigma^2)$, which is the multivariate normal distribution
\begin{eqnarray}
 N(y^s(x_{new},\theta_0)+\Sigma_1^\text{T}\Sigma_\gamma^{-1}\delta(\mathbf{x}), \tau^2(C_\gamma(x_{new},x_{new}) -\Sigma_1^\text{T}\Sigma_\gamma^{-1}\Sigma_1)+\sigma^2).\label{pred_mean}
\end{eqnarray}

\subsection{The Case of Expensive Computer Simulations}\label{Sec:expensive}

When the computer code is expensive to run, it is intractable to run MCMC based on (\ref{cheap}) directly. Instead, we need a \textit{surrogate} model to approximate the computer output function $y^s(\cdot,\cdot)$. In this setting \cite{kennedy2001bayesian} use the Gaussian process modeling again. Suppose we first run the computer simulation over a set of design points $\{(x_1^s,\theta_1^s),\ldots,(x_l^s,\theta_l^s)\}\subset\Omega\times\Theta$. We choose a Gaussian process with mean $m_\beta(\cdot)$ and covariance function $\tau'^2 C^s_{\gamma'}(\cdot,\cdot)$ as the prior for $y^s$, where $\beta,\tau'$ and $\gamma'$ are hyper-parameters. Besides, the prior processes of $y^s$ and $\delta$ are assumed to be independent.

The Bayesian analysis for the present model is similar to that in Section \ref{Sec:cheap} but with more cumbersome derivations.
We write $\mathbf{y}^s:=(y^s(x_1^s,\theta_1^s),\ldots,y^s(x_l^s,\theta_l^s))^\text{T}$ and define $(n+l)$-dimensional vectors
\begin{eqnarray*}
\mathbf{x}^E=(x^E_1,\ldots,x^E_{n+l})^\text{T}&:=&(x_1,\ldots,x_n,x_1^s,\ldots,x_l^s)^\text{T},\\
\mathbf{\theta}^E=(\theta^E_1,\ldots,\theta^E_{n+l})^\text{T}&:=&(\theta_0,\ldots,\theta_0,\theta_1^s,\ldots,\theta_l^s)^\text{T}.
\end{eqnarray*}
By (\ref{NP}) and (\ref{KO}), the joint distribution of $\mathbf{y}^p$ and $\mathbf{y}^s$ conditional on $\theta_0,\sigma^2,\gamma,\beta$ and $\tau$ is
\begin{eqnarray*}
(\mathbf{y}^p,\mathbf{y}^s)|\sigma^2,\gamma,\beta,\tau^2,\tau'^2,\gamma'\sim N\left(
m_\beta(\mathbf{x}^E),\Sigma_E+
\left(
  \begin{array}{cc}
    \Sigma_{1 1}+\sigma^2 I_n & 0 \\
    0 & 0 \\
  \end{array}
\right)
\right),
\end{eqnarray*}
where $m_\beta(\mathbf{x}^E)=(m_\beta(x_1^E),\ldots,m_\beta(x_{n+l}^E))^\text{T}$ and
\begin{eqnarray*}
\Sigma_E&=&\left(\tau'^2 C^s_{\gamma'}\left(\left(x_i^E,\theta_i^E\right),\left(x_j^E,\theta_j^E\right)\right)\right )_{i j},\\
\Sigma_{1 1}&=&\left(\tau^2 C_\gamma\left(x_i,x_j\right)\right)_{i j}.
\end{eqnarray*}
Then the posterior distribution of the parameters is given by
\begin{eqnarray*}
\pi(\theta_0,\sigma^2,\gamma,\beta,\tau^2,\tau'^2,\gamma'|\mathbf{y}^p,\mathbf{y}^s)\propto \pi(\mathbf{y}^p,\mathbf{y}^s|\theta_0,\sigma^2,\gamma,\beta,\tau^2,\tau'^2,\gamma') \pi(\theta_0,\sigma^2,\gamma,\beta,\tau^2,\tau'^2,\gamma').
\end{eqnarray*}
The parameter estimation proceeds in a similar manner to the MCMC scheme discussed in Section \ref{Sec:cheap}. As before, the prediction for the true process can be done along with the MCMC iterations. Noting the fact that $(y^p(x_{new}),\mathbf{y}^p,\mathbf{y}^s)$ follows a multivariate normal distribution given the model parameters, the posterior predictive distribution of $y^p(x_{new})$ can be obtained using the Bayes' theorem.

It can be seen that the modeling and analysis for the KO method with expensive computer code is much more complicated than that with cheap computer code. For the ease of mathematical analysis, our theoretical studies in the next section considers only the cases with cheap code. Hence, we omit the detailed formulae of the posterior density of the model parameters and the posterior predictive distribution of $y^p(x_{new})$ in this section.

\section{Theoretical Studies}\label{Sec:Theory}

In this section we conduct some theoretical study on the power of prediction of the KO method. For the ease of the mathematical treatment, we only consider the case of cheap computer code, because the formulae for the case of expensive computer code are much more complicated and cumbersome as shown in Section \ref{Sec:expensive}. We believe that this simplification does not affect our general conclusion.

The mathematical treatment to develop the asymptotic theory for the KO method also depend on the choice of the correlation family $C_\gamma$. In the present work, we restrict ourselves with the Mat\'{e}rn family of kernel functions \citep{stein1999interpolation}, defined as
\begin{eqnarray}\label{matern}
C_{\upsilon,\gamma}(s,t)=\frac{1}{\Gamma(\upsilon)2^{\upsilon-1}}\left(2\sqrt{\upsilon}\gamma \|s-t\|\right)^\upsilon K_\upsilon\left(2\sqrt{\upsilon}\gamma\|s-t\|\right),
\end{eqnarray}
where $K_\upsilon$ is the modified Bessel function of the second kind. In Mat\'{e}rn family, the model parameter $\upsilon$ dominates the smoothness of the process and $\gamma$ is a scale parameter. Because the smoothness parameter $\upsilon$ has an effect on the rate of convergence of the prediction, for simplicity we suppose $\upsilon$ is \textit{fixed} in the entire data analysis.

All proofs in this section are postponed to Appendix \ref{App:proof}.

\subsection{A Function Approximation Perspective}\label{Sec:Approximation}

In this section we follow the theoretical framework of \cite{tuo2014calibration} to study the prediction performance of the KO method. Under this framework, the physical responses are assumed to have no random error, i.e., $\epsilon_i$'s in (\ref{NP}) are zero. This is an unrealistic assumption in practice. But this assumption simplifies the model structure, so that we are able to find some mathematical tools which help us to understand certain intrinsic properties of the KO method.

From (\ref{NP}), we have $y_i^p=\zeta(x_i)$. We remind that $\zeta$ is indeed a deterministic function (as the expectation of the physical response). Therefore, we will regard the Gaussian process modeling technique used in the KO method as a way of reconstructing the function $\zeta$ based on samples $\zeta(x_i)$.

An immediate consequence of the deterministic assumption is $\delta(\mathbf{x})=\mathbf{y}^p-y^s(\mathbf{x},\theta_0)$, i.e., $\delta(\mathbf{x})$ is determined by $\theta_0$ given the observations. Thus (\ref{cheap}) is not applicable. Instead, we have
\begin{eqnarray*}
&&\pi(\theta_0,\tau^2,\gamma|\mathbf{y}^p)\propto \pi(\mathbf{y}^p|\theta_0,\tau^2,\gamma)\pi(\theta_0,\tau^2,\gamma)\nonumber\\
&\propto&(\det \Sigma_\gamma)^{-1/2}\exp\left\{-\frac{1}{2}(\mathbf{y}^p-y^s(\mathbf{x},\theta_0))^\text{T} \Sigma_\gamma^{-1} (\mathbf{y}^p-y^s(\mathbf{x},\theta_0)\right\}\pi(\theta_0,\tau^2,\gamma).
\end{eqnarray*}

To differentiate between the true process $\zeta$ and its estimate based on the observations, we denote a draw from the predictive distribution $\pi(\zeta(x_{new}))$ by $\zeta^{\text{rep}}(x_{new})$.
Then the posterior predictive distribution $\pi(\zeta^{\text{rep}}(x_{new})|\theta_0,\gamma,\mathbf{y}^p)$ is
\begin{eqnarray}
N\left(y^s(x_{new},\theta_0)+\Sigma_1^\text{T}\Sigma_\gamma^{-1}(\mathbf{y}^p-y^s(\mathbf{x},\theta_0)), \tau^2(C_{\upsilon,\gamma}(x_{new},x_{new})- \Sigma_1^\text{T}\Sigma_\gamma^{-1}\Sigma_1)\right).\label{posterior}
\end{eqnarray}

We now suppose the prior distribution $\pi(\theta_0,\tau^2,\gamma)$ is separable, i.e., $\pi(\theta_0,\tau^2,\gamma)=\pi(\theta_0)\pi(\tau^2)\pi(\gamma)$. Let $S_\theta,S_{\tau^2}$ and $S_\gamma$ denote the supports of the distributions $\pi(\theta_0),\pi(\tau^2)$ and $\pi(\gamma)$ respectively. For the ease of mathematical treatment, we further suppose that $S_\theta$ is a compact subset of $\mathbf{R}$, and $S_{\tau^2}\subset[0,\tau_0^2],S_\gamma\subset[\gamma_1,\gamma_2]$ for some $0<\tau_0^2<+\infty$, $0<\gamma_1<\gamma_2<+\infty$. The independence assumption of the prior distributions can be replaced with a more general assumption, which would not affect the validity of our theoretical analysis. However, the compact support assumption is technically unavoidable in the current treatment. Because here we only focus on the posterior mode, the use of the compact support assumption does not affect the practical applicability of the results. 

The aim of this section is to study the asymptotic behavior of
\begin{eqnarray*}
\hat{\mu}_{\theta,\gamma}&=&y^s(x_{new},\theta)+\Sigma_1^\text{T}\Sigma_\gamma^{-1}(\mathbf{y}^p-y^s(\mathbf{x},\theta)),\\
\hat{\varsigma}^2_{\tau^2,\gamma}&=&\tau^2(C_{\upsilon,\gamma}(x_{new},x_{new})-\Sigma_1^\text{T}\Sigma_\gamma^{-1}\Sigma_1),
\end{eqnarray*}
as the design points become dense in $\Omega$, for $(\theta,\tau^2,\gamma)\in S_\theta,S_{\tau^2},S_\gamma$. Clearly, the true posterior mean of $\zeta^{\text{rep}}(x_{new})$ given by (\ref{posterior}) is
\begin{eqnarray*}
E[\zeta^{\text{rep}}(x_{new})|\mathbf{y}^p]=E[\hat{\mu}_{\hat{\theta},\hat{\gamma}}|\mathbf{y}^p],
\end{eqnarray*}
where $(\hat{\theta},\hat{\gamma})$ follows the posterior distribution $\pi(\theta_0,\gamma|\mathbf{y}^p)$. Note that
\begin{eqnarray*}
&&|E[\zeta^{\text{rep}}(x_{new})|\mathbf{y}^p]-\zeta(x_{new})|\\
&=& \left|E\left\{E[\zeta^{\text{rep}}(x_{new})-\zeta(x_{new})|\mathbf{y}^p,\hat{\theta},\hat{\gamma}] \big| \mathbf{y}^p\right\}\right|\\
&\leq&\sup_{\theta\in S_\theta,\gamma\in S_\gamma} \left|E[\zeta^{\text{rep}}(x_{new})-\zeta(x_{new})|\mathbf{y}^p,\theta,\gamma]\right| \\
&=&\sup_{\theta\in S_\theta,\gamma\in S_\gamma}|\hat{\mu}_{\theta,\gamma}-\zeta(x_{new})|,
\end{eqnarray*}
i.e., the bias of the posterior predictive mean can be bounded by the supremum of $|\hat{\mu}_{\theta,\gamma}-\zeta(x_{new})|$. Similarly, we find
\begin{eqnarray*}
\text{Var}(\zeta^{\text{rep}}(x_{new})|\mathbf{y}^p)\leq \sup_{\tau^2\in S_{\tau^2},\gamma\in S_\gamma}\hat{\varsigma}^2_{\tau^2,\gamma}.
\end{eqnarray*}

In this section we will bound $\sup_{\theta\in S_\theta,\gamma\in S_\gamma}|\hat{\mu}_{\theta,\gamma}-\zeta(x_{new})|$ and $\sup_{\tau^2\in S_{\tau^2}\gamma\in S_\gamma}\hat{\varsigma}^2_{\tau^2,\gamma}$.
To this end, we resort to the theory of native spaces. We refer to \cite{wendland2005scattered} for detailed discussions. For a symmetric and positive definite function $\Phi$ over $\Omega\times\Omega$, consider the linear space
\begin{eqnarray*}
F_\Phi(\Omega):=\left\{\sum_{i=1}^m \alpha_i\Phi(s_i,\cdot):m\in\mathbb{N}^+,\alpha_i\in\mathbf{R}\right\},
\end{eqnarray*}
equipped with the inner product
\begin{eqnarray}
\left\langle\sum_{i=1}^m \alpha_i\Phi(s_i,\cdot),\sum_{j=1}^l \beta_j\Phi(t_j,\cdot)\right\rangle=\sum_{i=1}^m\sum_{j=1}^l \alpha_i\beta_j\Phi(s_i,t_j).\label{innerproduct}
\end{eqnarray}
The completion of $F_\Phi(\Omega)$ with respect to its inner product is called the native space generated by $\Phi$, denoted by $\mathcal{N}_{\Phi}(\Omega)$. Denote the inner product and the norm of $\mathcal{N}_\Phi(\Omega)$ by $\langle\cdot,\cdot\rangle_{\mathcal{N}_\Phi(\Omega)}$ and $\|\cdot\|_{\mathcal{N}_\Phi(\Omega)}$ respectively.

Now we state the interpolation scheme in the native space. Let $f\in\mathcal{N}_\Phi(\Omega)$ and $\mathbf{x}=\{x_1,\ldots,x_n\}$ a set of distinct points in $\Omega$. Let $\mathbf{y}=(f(x_1),\ldots,f(x_n))^\text{T}$ be the observerd data. Define
\begin{eqnarray}
s_{f,\mathbf{x}}(x)=\sum_{i=1}^n u_i\Phi(x_i,x),\label{interpolant}
\end{eqnarray}
where $u=(u_1,\ldots,u_n)^\text{T}$ is given by the linear equation
\begin{eqnarray*}
\mathbf{y}=\Phi(\mathbf{x},\mathbf{x})u
\end{eqnarray*}
for $(\Phi(\mathbf{x},\mathbf{x}))_{i j}=\Phi(x_i,x_j)$.

Clearly, $s_{f,\mathbf{x}}\in F_\Phi$ and thus $s_{f,\mathbf{x}}\in \mathcal{N}_\Phi(\Omega)$.
The next lemma can be found in \cite{wendland2005scattered}. For the completeness of the present article, we provide its proof in Appendix \ref{App:proof}.

\begin{lemma}\label{lemma:1}
For $f\in\mathcal{N}_\Phi(\Omega)$ and a set of design points $\mathbf{x}\subset\Omega$,
\begin{eqnarray*}
\langle s_{f,\mathbf{x}},f-s_{f,\mathbf{x}}\rangle_{\mathcal{N}_\Phi(\Omega)}=0.
\end{eqnarray*}
\end{lemma}

From Lemma \ref{lemma:1} we can deduce the Pythagorean identity
\begin{eqnarray}
\|s_{f,\mathbf{x}}\|^2_{\mathcal{N}_\Phi(\Omega)}+\|f-s_{f,\mathbf{x}}\|^2_{\mathcal{N}_\Phi(\Omega)} =\|f\|^2_{\mathcal{N}_\Phi(\Omega)}.\label{pathagorean}
\end{eqnarray}
Now we consider an arbitrary function $h\in\mathcal{N}_\Phi(\Omega)$ which interpolates $f$ over $\mathbf{x}$, denoted as $f|_\mathbf{x}=h|_\mathbf{x}$. Then we have $s_{f,\mathbf{x}}=s_{h,\mathbf{x}}$ and thus (\ref{pathagorean}) also holds true if we replace $f$ with $h$. This suggests $\|s_{f,\mathbf{x}}\|_{\mathcal{N}_\Phi(\Omega)}\leq \|h\|_{\mathcal{N}_\Phi(\Omega)}$,
which yields the following optimality condition
\begin{eqnarray}
s_{f,\mathbf{x}}=\operatorname*{argmin}_{\substack{h\in\mathcal{N}_\Phi(\Omega)\\h|_\mathbf{x}=f|_\mathbf{x}}} \|h\|_{\mathcal{N}_\Phi(\Omega)},\label{optimality}
\end{eqnarray}
i.e., $s_{f,\mathbf{x}}$ has the minimum native norm among all functions in $\mathcal{N}_\Phi(\Omega)$ that interpolate $f$ over $\mathbf{x}$.

It can be shown that the native space generated by the Mat\'{e}rn kernel $C_{\upsilon,\gamma}$ for $\upsilon\geq 1$ coincides with the (fractional) Sobolev space $H^{\upsilon+d/2}(\Omega)$ \citep{adams2003sobolev}, and the norms are equivalent. See \cite{tuo2014calibration} for details. Moreover, we can also prove that the norms of the native spaces generated by $C_{\upsilon,\gamma}$ for a set of $\gamma$ values bounded away from 0 and $+\infty$ are equivalent.

\begin{lemma}\label{th:equivalence}
Suppose $\upsilon\geq 1$. There exist constants $c_1,c_2>1$, so that
\begin{eqnarray}
c_1\|f\|_{H^{\upsilon+d/2}(\Omega)}\leq \|f\|_{\mathcal{N}_{C_{\upsilon,\gamma}}(\Omega)}\leq c_2\|f\|_{H^{\upsilon+d/2}(\Omega)}
\end{eqnarray}
holds for all $f\in H^{\upsilon+d/2}(\Omega)$ and all $\gamma\in [\gamma_1,\gamma_2]$.
\end{lemma}

Next, we turn to the error estimate of the interpolant $s_{f,\mathbf{x}}$.
\cite{wendland2005scattered} shows that for $u\in H^\mu(\Omega)$ with $u|_\mathbf{x}=0$ and $\lfloor\mu\rfloor>d/2$,
\begin{eqnarray*}
\|u\|_{L_\infty(\Omega)}\leq C h_{\mathbf{x},\Omega}^{\mu-d/2}\|u\|_{H^\mu(\Omega)},
\end{eqnarray*}
provided that $\mathbf{x}$ is ``sufficiently dense'',
where $C$ is independent of $\mathbf{x}$ and $u$; $h_{\mathbf{x},\Omega}$ is the fill distance of the design $\mathbf{x}$ defined as
\begin{eqnarray*}
h_{\mathbf{x},\Omega}=\sup_{x\in\Omega}\min_{x_j\in \mathbf{x}}\|x-x_j\|.
\end{eqnarray*}
Here ``$\mathbf{x}$ is sufficiently dense'' means that its fill distance $h_{\mathbf{x},\Omega}$ is less than a constant $h_0$ depending only on $\Omega$ and $\mu$.
Noting the fact that $(f-s_{f,\mathbf{x}})|_\mathbf{x}=0$ and $f-s_{f,\mathbf{x}}\in H^{\upsilon+d/2}(\Omega)$, we obtain that for $\upsilon\geq 1$,
\begin{eqnarray*}
\|f-s_{f,\mathbf{x}}\|_{L_\infty(\Omega)}\leq C h_{\mathbf{x},\Omega}^{\upsilon}\|f-s_{f,\mathbf{x}}\|_{H^{\upsilon+d/2}(\Omega)},
\end{eqnarray*}
which, together with (\ref{pathagorean}), yields
\begin{eqnarray}
\|f-s_{f,\mathbf{x}}\|_{L_\infty(\Omega)}\leq C h_{\mathbf{x},\Omega}^{\upsilon}\|f\|_{H^{\upsilon+d/2}(\Omega)}.\label{bound}
\end{eqnarray}
Then we apply Lemma \ref{th:equivalence} to prove Lemma \ref{th:bound}.

\begin{lemma}\label{th:bound}
Suppose $\upsilon\geq 1$. For $f\in H^{\upsilon+d/2}(\Omega)$, let $s_{f,\mathbf{x}}$ be the interpolant of $f$ over $\mathbf{x}$ with the kernel $C_{\gamma,\upsilon}$, $\gamma\in[\gamma_1,\gamma_2]$. Then for sufficiently dense $\mathbf{x}$
\begin{eqnarray*}
\|f-s_{f,\mathbf{x}}\|_{L_\infty(\Omega)}\leq C h_{\mathbf{x},\Omega}^{\upsilon}\|f\|_{\mathcal{N}_{C_{\upsilon,\gamma}}(\Omega)},
\end{eqnarray*}
where $C$ is independent of the choices of $f$, $\mathbf{x}$ and $\gamma$.
\end{lemma}


Following the notation of \cite{tuo2014calibration}, we define $\epsilon(x,\theta)=\zeta(x)-y^s(x,\theta)$. It is commented by \cite{tuo2014calibration} that in general $\theta_0$ is not estimable due to the identifiability problem, and thus neither is $\delta(\cdot)=\epsilon(\cdot,\theta_0)$. However, as will be shown later, the function $\epsilon(\cdot,\cdot)$ can be consistently estimated using KO calibration. Suppose $\epsilon(\cdot,\theta)\in H^{\upsilon+d/2}(\Omega)$ for each $\theta\in S_\theta$. Let $\epsilon(\mathbf{x},\theta)=(\epsilon(x_1,\theta),\ldots,\epsilon(x_n,\theta))^\text{T}$. Clearly, $\mathbf{y}^p-y^s(\mathbf{x},\theta)=\epsilon(\mathbf{x},\theta)$ and thus
\begin{eqnarray*}
s_{\epsilon(\cdot,\theta),\mathbf{x}}(x_{new})=\Sigma_1^\text{T}\Sigma_\gamma^{-1}(\mathbf{y}^p-y^s(\mathbf{x},\theta)).
\end{eqnarray*}
By (\ref{bound}) we obtain
\begin{eqnarray}
|\hat{\mu}_{\theta,\gamma}-\zeta(x_{new})|&=&|\epsilon(x_{new},\theta)-s_{\epsilon(\cdot,\theta),\mathbf{x}}(x_{new})|\nonumber\\
&\leq& C h^\upsilon_{\mathbf{x},\Omega}\|\epsilon(\cdot,\theta)\|_{H^{\upsilon+d/2}(\Omega)}\nonumber\\
&\leq&C h^\upsilon_{\mathbf{x},\Omega}\sup_{\theta\in S_\theta}\|\epsilon(\cdot,\theta)\|_{H^{\upsilon+d/2}(\Omega)}.\label{bmu}
\end{eqnarray}

The error bound for the variance term can be obtained similarly. Elementary calculations show that
\begin{eqnarray*}
\Sigma_1^\text{T}\Sigma_\gamma^{-1}\Sigma_1=s_{C_{\upsilon,\gamma}(\cdot,x_{new}),\mathbf{x}}(x_{new}).
\end{eqnarray*}
Hence we apply Lemma \ref{th:bound} to find
\begin{eqnarray}
|\tau^2(C_{\upsilon,\gamma}(x_{new},x_{new})-\Sigma_1^\text{T}\Sigma_\gamma^{-1}\Sigma_1)|&=&\tau^2|C_{\upsilon,\gamma}(x_{new},x_{new})- s_{C_{\upsilon,\gamma}(\cdot,x_{new}),\mathbf{x}}(x_{new})|\nonumber\\
&\leq&\tau^2_0 C h^\upsilon_{\mathbf{x},\Omega}\|C_{\upsilon,\gamma}(\cdot,x_{new})\|_{\mathcal{N}_{C_{\upsilon,\gamma}}(\Omega)}\nonumber\\
&=&\tau^2_0 C h^\upsilon_{\mathbf{x},\Omega},\label{bsigma}
\end{eqnarray}
where the last equality follows from the fact that $\|C_{\upsilon,\gamma}(\cdot,x_{new})\|_{\mathcal{N}_{C_{\upsilon,\gamma}}(\Omega)}=1$. We summarize our findings in (\ref{bmu}) and (\ref{bsigma}) as Theorem \ref{th:deterministic}.

\begin{theorem}\label{th:deterministic}
Suppose $\upsilon\geq 1,\gamma\in[\gamma_1,\gamma_2],\tau\leq \tau_0$. Then for a sufficiently dense design $\mathbf{x}$, we have the upper bound for the predictive mean as
\begin{eqnarray*}
\sup_{\theta\in S_\theta,\gamma\in S_\gamma}|\hat{\mu}_{\theta,\gamma}-\zeta(x_{new})|\leq C h^\upsilon_{\mathbf{x},\Omega}\sup_{\theta\in S_\theta}\|\epsilon(\cdot,\theta)\|_{H^{\upsilon+d/2}(\Omega)},
\end{eqnarray*}
and the upper bound for the predictive variance as
\begin{eqnarray*}
\sup_{\tau^2\in S_\tau^2,\gamma\in S_\gamma}\hat{\varsigma}^2_{\tau^2,\gamma}\leq \tau^2_0 C h^\upsilon_{\mathbf{x},\Omega},
\end{eqnarray*}
with a constant $C$ depending only on $\Omega, \upsilon,\gamma_1,\gamma_2$.
\end{theorem}

From Theorem \ref{th:deterministic}, the rate of convergence is $O(h^\upsilon_{\mathbf{x},\Omega})$, which is known to be optimal in the current setting \citep{wendland2005scattered}.
It is worth noting that the predictive behavior of the KO calibration is more robust than in the case of estimation as shown by \cite{tuo2014calibration}, Theorem 4.2. Specifically, they show the KO calibration estimator tends to the minimizer of a norm involving the prior assumption, i.e., the KO calibration can reply heavily on the prior specification. By comparison, the predictive performance as shown in Theorem \ref{th:deterministic} above does not depend on the choice of the prior asymptotically.

\subsection{A Nonparametric Regression Perspective}\label{Sec:Nonparametric}

Now we turn to a more realistic case, where the physical observations have random measurement errors. As before, we treat the true process $\zeta(\cdot)$ as a deterministic function. For the ease of mathematical treatment, in this section we fix the value of $\gamma$. Our analysis later will show in Theorems \ref{th:convergencerate} and \ref{th:variance} that the resulting rate of convergence is not influenced by the choice of $\gamma$. Other parameters are either estimated or chosen to vary along with the sample size $n$. 

To study the predictive behavior of the KO method asymptotically, the key is to understand the posterior mode of $\delta(\mathbf{x})$ in (\ref{cheap}). We first introduce the representer theorem \citep{scholkopf2001generalized,wahba1990spline}. We also give a proof of the representer theorem using Lemma \ref{lemma:1} in Appendix \ref{App:proof}.

\begin{lemma}[Representer Theorem]\label{th:representer}
Let $x_1,\ldots,x_n$ be a set of distinct points in $\Omega$ and $L:\mathbf{R}^n\rightarrow \mathbf{R}$ be an arbitrary function. Denote the minimizer of the optimization problem
\begin{eqnarray*}
\min_{f\in\mathcal{N}_\Phi(\Omega)}L(f(x_1),f(x_2),\ldots,f(x_n))+\|f\|^2_{\mathcal{N}_\Phi(\Omega)}
\end{eqnarray*}
by $\hat{f}$. Then $\hat{f}$ possesses the representation
\begin{eqnarray*}
\hat{f}=\sum_{i=1}^n \alpha_i \Phi(x_i,\cdot),
\end{eqnarray*}
with coefficients $\alpha_i\in\mathbf{R}, i=1,\ldots,n$.
\end{lemma}

Similar to Section \ref{Sec:Approximation}, we first fix the values of $\tau^2,\sigma^2,\gamma$ in their domain. Then we consider the profile posterior density function of $\delta(\mathbf{x})$, which, according to (\ref{cheap}), is proportional to
\begin{eqnarray}
\pi_{\tau^2,\sigma^2,\gamma}(\theta,\delta(\mathbf{x}))= \exp\left\{-\frac{1}{2\sigma^2}\|\mathbf{y}^2-y^s(\mathbf{x},\theta)-\delta(\mathbf{x})\|^2- \frac{\delta(\mathbf{x})^\text{T}\Sigma_\gamma^{-1}\delta(\mathbf{x})}{2\tau^2}\right\}.\label{delta}
\end{eqnarray}

The profile posterior mode $(\hat{\theta}_{KO},\hat{\delta}(\mathbf{x}))$ maximizes $\pi_{\tau^2,\sigma^2,\gamma}(\cdot,\cdot)$. Using the representer theorem, we show an equality between $\hat{\delta}(\mathbf{x})$ and the solution to a penalized least squares problem.

\begin{theorem}\label{th:leastsquares}
Let $(\hat{\theta},\hat{\Delta})$ be the solution to
\begin{eqnarray}
\operatorname*{argmin}_{\substack{\theta\in\Theta\\f\in\mathcal{N}_{C_{\upsilon,\gamma}}(\Omega)}}\sum_{i=1}^n (y_i^p-y^s(x_i,\theta)-f(x_i))^2+\frac{\sigma^2}{\tau^2} \|f\|^2_{\mathcal{N}_{C_{\upsilon,\gamma}}(\Omega)}.\label{epsilon}
\end{eqnarray}
Then $\hat{\theta}=\hat{\theta}_{KO}$ and $(\hat{\Delta}(x_1),\ldots,\hat{\Delta}(x_n))^\text{T}=:\hat{\Delta}(\mathbf{x})=\hat{\delta}(\mathbf{x})$.
\end{theorem}

Now we are ready to state the main asymptotic theory. We will first investigate the asymptotic properties of the predictive mean. Next we will consider the consistency of the predictive variance.

From (\ref{pred_mean}), the predictive mean of the KO model is
\begin{eqnarray}
\zeta^{\text{rep}}(x_{new})=y^s(x_{new},\hat{\theta}_{KO})+\Sigma_1^\text{T}(x_{new})\Sigma^{-1}_\gamma\hat{\delta}(\mathbf{x}),\label{zetahat}
\end{eqnarray}
where $\Sigma_1(x_{new})=(C_{\upsilon,\gamma}(x_{new},x_1),\ldots,C_{\upsilon,\gamma}(x_{new},x_n))^\text{T}$.
Invoking Theorem \ref{th:leastsquares}, we have $\hat{\delta}(\mathbf{x})=\hat{\Delta}(\mathbf{x})$ with $\hat{\Delta}$ defined in (\ref{epsilon}).
Using the formula of the kernel interpolant given by (\ref{interpolant}), it can be seen that $\hat{\zeta}(\cdot)-y^s(\cdot,\hat{\theta}_{KO})$ is the kernel interpolant of the data $(\mathbf{x},\hat{\Delta}(\mathbf{x}))$. Hence, from Lemma \ref{th:representer} and Theorem \ref{th:leastsquares} we have
\begin{eqnarray*}
\hat{\zeta}(\cdot)-y^s(\cdot,\hat{\theta}_{KO})=\hat{\Delta}(\cdot).
\end{eqnarray*}

We note that the ratio of the variances $\sigma^2/\tau^2$ plays an important role in (\ref{epsilon}). In the nonparametric regression literature, such a quantity is commonly referred to as the \textit{smoothing parameter}. The smoothing parameter is a tuning parameter to balance the bias and variance of the estimator. It can be seen that as $\sigma^2/\tau^2\rightarrow\infty$, $\hat{\epsilon}$ tends to 0, which has the smallest variance but a large bias; as $\sigma^2/\tau^2\downarrow 0$, $\hat{\epsilon}$ will eventually interpolate $(x_i,y_i^p-y^s(x_i,\hat{\theta}_{KO}))$, which typically leads to an over-fitting problem. We denote $\sigma^2/\tau^2$ by $r_n$ when the sample size is $n$. According to Theorem \ref{th:convergencerate}, the optimal rate for $r_n$ is $r_n\sim n^{\frac{4\upsilon+5d}{4\upsilon+4d}}$. There is a theory by \cite{van2008rates} which says that the optimal tuning rate can be automatically achieved by following a standard Bayesian analysis procedure. To save space, we do not pursue this approach here.

Some asymptotic theory for the penalized least squares problem (\ref{epsilon}) is available in the literature \citep{vandegeer2000empirical}. In order to employ such a theory, we need to choose the smoothing parameter $r_n$ to diverge at an appropriate rate as $n$ goes to infinity.
For convenience, we suppose that the design points are randomly chosen. We consider the rate of convergence of the penalized least squares estimator under the $L_2$ metric. We assume that $y^s$ is Lipschitz continuous. Then the metric entropy of $\{y^s(\cdot,\theta):\theta\in\Theta\}$ is dominated by that of the unit ball of the nonparametric class $\mathcal{N}_{C_{\upsilon,\gamma}}(\Omega)$. See \cite{van1996weak} for details. Theorem \ref{th:convergencerate} then is a direct consequence of Theorem 10.2 of \cite{vandegeer2000empirical}, where the required upper bound for the metric entropy is obtained from (3.6) of \cite{tuo2015efficient}.

\begin{theorem}\label{th:convergencerate}
Suppose the design points $\{x_i\}$ are independent samples from the uniform distribution over $\Omega$. We assume $\upsilon\geq 1$ and $y^s$ is Lipschitz continuous. Choose $r_n$ appropriately so that $r_n\sim n^{\frac{4\upsilon+5d}{4\upsilon+4d}}$.
Under model (\ref{NP}) with $\sigma^2>0$, the KO predictor $\hat{\zeta}$ defined in (\ref{zetahat}) has the approximation properties 
\begin{eqnarray}
\frac{1}{n}\sum_{i=1}^n(\hat{\zeta}(x_i)-\zeta(x_i))^2=O_p(n^{-\frac{2\upsilon+d}{2\upsilon+2d}}),\label{rateempirical}\\
\|\hat{\zeta}-\zeta\|_{L_2(\Omega)}=O_p(n^{-\frac{\upsilon+d/2}{2\upsilon+2d}}),\label{rate1}
\end{eqnarray}
and
\begin{eqnarray}
\|\hat{\zeta}-\zeta\|_{H^{\upsilon+d/2}(\Omega)}=O_p(1).\label{rate2}
\end{eqnarray}
\end{theorem}

Since the native space $\mathcal{N}_{C_{\upsilon,\gamma}(\Omega)}$ is equivalent to the Sobolev space $H^{\upsilon+d/2}(\Omega)$, the rate of convergence in (\ref{rate1}) is optimal according to the theory by \cite{stone1982optimal}. The interpolation inequality of the Sobolev spaces (see (12) in Chapter 5 of \cite{adams2003sobolev}) claims that
\begin{eqnarray*}
\|\hat{\zeta}-\zeta\|_{L_\infty(\Omega)}\leq K\|\hat{\zeta}-\zeta\|_{H^{\upsilon+d/2}(\Omega)}^{\frac{d}{2(\upsilon+d/2)}} \|\hat{\zeta}-\zeta\|_{L_2(\Omega)}^{1-\frac{d}{2(\upsilon+d/2)}},
\end{eqnarray*}
with constant $K$ depending only on $\Omega$ and $\upsilon$. In view of (\ref{rate1}) and (\ref{rate2}), we have
\begin{eqnarray}
\|\hat{\zeta}-\zeta\|_{L_\infty(\Omega)}=O_p(n^{-\frac{\upsilon}{2\upsilon+2d}}),\label{rateuniform}
\end{eqnarray}
which gives the rate of convergence of the predictive mean under the uniform metric.

Now we turn to the consistency of the predictive variance. To avoid ambiguity, we denote the true value of $\sigma^2$ by $\sigma^2_0$. First we consider the profile posterior mode of $\sigma^2$ in (\ref{cheap}). For simplicity, we only consider the non-informative prior for $\sigma^2$ with $\pi(\sigma^2)\propto 1$. However, we note that the limiting value of the posterior mode of $\sigma^2$ is not affected by the choice of $\pi(\sigma^2)$, provided that $\sigma^2_0$ is contained in the support of $\pi(\sigma^2)$. It is easily seen from (\ref{cheap}) that the posterior mode of $\sigma^2$ is
\begin{eqnarray*}
\hat{\sigma}^2&=&\frac{\|\mathbf{y}^p-y^s(\mathbf{x},\hat{\theta}_{KO})-\hat{\delta}(\mathbf{x})\|^2}{n}\\
&=&\frac{1}{n}\sum_{i=1}^n\{e_i+(\hat{\zeta}(x_i)-\zeta(x_i))\}^2\\
&=&\frac{1}{n}\sum_{i=1}^n e_i^2+\frac{2}{n}\sum_{i=1}^n e_i(\hat{\zeta}(x_i)-\zeta(x_i))+ \frac{1}{n}\sum_{i=1}^n(\hat{\zeta}(x_i)-\zeta(x_i))^2,
\end{eqnarray*}
which yields the following asymptotic property
\begin{eqnarray}
|\hat{\sigma}^2-\sigma_0^2|&=&\left|\hat{\sigma}^2-\frac{1}{n}\sum_{i=1}^n e_i^2+O_p(n^{-1/2})\right|\nonumber\\
&=&\left|\frac{2}{n}\sum_{i=1}^n e_i(\hat{\zeta}(x_i)-\zeta(x_i))+ \frac{1}{n}\sum_{i=1}^n(\hat{\zeta}(x_i)-\zeta(x_i))^2+ O_p(n^{-1/2})\right|\nonumber\\
&\leq&2\left(\frac{1}{n}\sum_{i=1}^n e_i^2\right)^{1/2}\left(\frac{1}{n}\sum_{i=1}^n(\hat{\zeta}(x_i)-\zeta(x_i))^2\right)^{1/2} \nonumber\\&&+ \frac{1}{n}\sum_{i=1}^n(\hat{\zeta}(x_i)-\zeta(x_i))^2+ O_p(n^{-1/2})\nonumber\\
&=&O_p(n^{-\frac{\upsilon+d/2}{2\upsilon+2d}}),
\end{eqnarray}
where the inequality follows from Cauchy-Schwarz inequality and the last equality follows from (\ref{rateempirical}).

From (\ref{pred_mean}), the predictive variance of the KO model is
\begin{eqnarray}
\hat{\varsigma}^2(x_{new})=\tau^2(C_\gamma(x_{new},x_{new})-\Sigma_1^\text{T}\Sigma^{-1}_\gamma\Sigma_1)+\hat{\sigma}^2.\label{variancerate}
\end{eqnarray}
As discussed in Section \ref{Sec:Approximation}, $C_\gamma(x_{new},x_{new})-\Sigma_1^\text{T}\Sigma^{-1}_\gamma\Sigma_1$ is the approximation error of the kernel interpolation for the function $C_\gamma(\cdot,x_{new})$. Clearly, the error from the interpolation problem discussed in Section \ref{Sec:Approximation} should be no more than that from the smoothing problem discussed in the current section, because of the presence of the random error in the latter situation. Thus we have
\begin{eqnarray}
&&\sup_{x_{new}\in\Omega}|\tau^2(C_\gamma(x_{new},x_{new})-\Sigma_1^\text{T}\Sigma^{-1}_\gamma\Sigma_1)|\nonumber\\
&=&\sup_{x_{new}\in\Omega}|r_n\hat{\sigma}^2(C_\gamma(x_{new},x_{new})-\Sigma_1^\text{T}\Sigma^{-1}_\gamma\Sigma_1)|\nonumber\\
&=&O_p(n^{-\frac{d}{4\upsilon+4d}}n^{-\frac{\upsilon}{2\upsilon+2d}})=O_p(n^{-\frac{\upsilon+d/2}{2\upsilon+2d}}),\label{variancerate2}
\end{eqnarray}
where the second equality follows from the assumption $r_n\sim n^{\frac{d}{4\upsilon+4d}}$ in Theorem \ref{th:convergencerate}, (\ref{variancerate}) and (\ref{rateuniform}). Combining (\ref{variancerate}) and (\ref{variancerate2}) we obtain Theorem \ref{th:variance}.

\begin{theorem}\label{th:variance}
Under the conditions of Theorem \ref{th:convergencerate}, we have the error bound for the predictive variance under the uniform metric
\begin{eqnarray*}
\|\hat{\varsigma}^2(\cdot)-\sigma^2_0\|_{L_\infty(\Omega)}=O_p(n^{-\frac{\upsilon+d/2}{2\upsilon+2d}}).
\end{eqnarray*}
\end{theorem}

Noting that $\sigma^2_0$ is the variance of the random noise, which is present in prediction for a new physical response. In other word, no predictor has a mean square error less than $\sigma^2_0$. Theorems \ref{th:convergencerate} and \ref{th:variance} reveal that the predictive distribution given by the KO method can capture the true uncertainty of the physical data in the asymptotic sense.


\section{Discussions}\label{Sec:discussions}

In this work, we prove some error bounds for the predictive error given by the Kennedy-O'Hagan method in two cases: 1) the physical observations have no random error and 2) the physical observations are noisy. For the ease of mathematical analysis, we only consider the Mat\'{e}rn correlation family. If a different covariance structure is used, we believe that the consistency for the predictive mean and the predictive variance still holds. However, additional study is required to obtain the appropriate rate of convergence. In our entire analysis, we ignore the estimation for some model parameters like $\gamma$ and $\tau^2$. One may consider the error estimate in a fully Bayesian procedure. But the analysis will then become rather complicated and it is unclear whether a new theory can be developed along the same lines.

Throughout this work, we assume that the smoothness parameter $\upsilon$ is given. From Theorems \ref{th:deterministic} and \ref{th:convergencerate}, a better rate of convergence can be obtained by using a greater $\upsilon$ provided that the target function still lies in $\mathcal{N}_{C_{\upsilon,\gamma}}(\Omega)$. Thus ideally, one should choose $\upsilon$ to be close to, but no more than the true degree of smoothness of the target function. There are different ways of choosing data-dependent $\upsilon$, but the mathematical analysis will become much more involved. We refer to \cite{loh2015estimating} and the references therein for some related discussions.

In Section \ref{Sec:Nonparametric}, we assume that the design points $x_i$'s are random samples over $\Omega$. In practice, one may also wish to choose design points using a systematic (deterministic) scheme. In general, if a sequence of fixed designs is used, the same (optimal) rate of convergence is retained, provided that these designs satisfy certain space-filling conditions. We refer to \cite{utreras1988convergence} for the results and necessary mathematical tools.

Finally, we discuss how the calibration procedure can have an effect on the prediction for the true process. In this article, we allow the number of physical measurements to grow to infinity and obtain the rate of convergence. By comparing the results presented here and the standard ones using radial basis functions or smoothing spline approximation, we find that the rate of convergence is not elevated by doing calibration. But we can use the following heuristics to show that by doing KO calibration the predictive error can be improved by a constant factor. To see this, we review the proof of Theorem \ref{th:deterministic},
from which it can be seen that if we fix $\Phi$ and $\gamma$, the predictive error is bounded by
\begin{eqnarray}
|\hat{\mu}_{\theta,\gamma}-\zeta(x_{new})|\leq C h^\upsilon_{\mathbf{x},\Omega}\|\epsilon(\cdot,\theta)\|_{H^{\upsilon+d/2}(\Omega)},\label{boundderterminisitic}
\end{eqnarray}
for an arbitrarily chosen $\theta\in\Theta$. So the rate of convergence is given by $O(h^\upsilon_{\mathbf{x},\Omega})$, and $\|\epsilon(\cdot,\theta)\|_{H^{\upsilon+d/2}(\Omega)}$ acts as a constant factor. \cite{tuo2014calibration} show that under certain conditions, the KO estimator for the calibration parameter converges to
\begin{eqnarray*}
\theta'=\operatorname*{argmin}_{\theta\in\Theta}\|\epsilon(\cdot,\theta)\|_{\mathcal{N}_\Phi(\Omega)},
\end{eqnarray*}
as the design points become dense over $\Omega$.
Since $\|\cdot\|_{\mathcal{N}_\Phi(\Omega)}$ is equivalent to $\|\cdot\|_{H^{\upsilon+d/2}(\Omega)}$, estimating the calibration parameter via the KO method is apparently beneficial for prediction in the sense that the upper error bound is reduced because $\|\epsilon(\cdot,\theta')\|_{\mathcal{N}_\Phi(\Omega)}\leq \|\epsilon(\cdot,\theta)\|_{\mathcal{N}_\Phi(\Omega)}$ for all $\theta\in\Theta$. There is a similar phenomenon for the stochastic case, by using the arguments in the proof of Theorem 10.2 of \cite{vandegeer2000empirical}.

\appendix

{\bf\Large\center{Appendix}}

\section{Technical Proofs}\label{App:proof}

In this section, we present the proofs for Lemma \ref{lemma:1}, Lemma \ref{th:equivalence}, Lemma \ref{th:representer} and Theorem \ref{th:leastsquares}.

\begin{proof}[Proof of Lemma \ref{lemma:1}]
We first assume $f\in F_\Phi$. If $f=s_{f,\mathbf{x}}$, there is nothing to prove. If $f\neq s_{f,\mathbf{x}}$, without loss of generality, we write
\begin{eqnarray*}
f(x)=\sum_{i=1}^{n+m} \alpha_i \Phi(x,x_i),
\end{eqnarray*}
for an extra set of distinct points $\{x_{n+1},\ldots,x_{n+m}\}\subset \Omega$. Now partition $(A_{i,j})=\Phi(x_i,x_j), 1\leq i,j\leq n+m$ into
\begin{eqnarray*}
A=\begin{pmatrix}
(A_1)_{n\times n} & (A_2)_{n\times m}\\
(A_3)_{m\times n} & (A_4)_{m\times m}
\end{pmatrix},
\end{eqnarray*}
where $A_3=A_2^\text{T}$ because $\Phi$ is symmetric.

Let $\mathbf{y}=(f(x_1),\ldots,f(x_n))^\text{T},a_1=(\alpha_1,\ldots,\alpha_n)^\text{T},a_2=(\alpha_{n+1},\ldots,\alpha_{n+m})^\text{T}$. Clearly, $\mathbf{y}=A_1 a_1+A_2 a_2$. By the definition of $s_{f,\mathbf{x}}$, we have
\begin{eqnarray*}
s_{f,\mathbf{x}}(x)=\sum_{i=1}^n u_i\Phi(x,x_i),
\end{eqnarray*}
with $u=(u_1,\ldots,u_n)^\text{T}$ satisfying $\mathbf{y}=A_1 u$. Then from (\ref{innerproduct}) we obtain
\begin{eqnarray}
&&\langle s_{f,\mathbf{x}},f-s_{f,\mathbf{x}}\rangle_{\mathcal{N}_\Phi(\Omega)}\nonumber\\
&=&\left\langle \sum_{i=1}^n u_i\Phi(x,x_i),\sum_{i=1}^n (\alpha_i-u_i)\Phi(x,x_i)+\sum_{i=n+1}^{n+m}\alpha_i\Phi(x,x_i)\right\rangle_{\mathcal{N}_\Phi(\Omega)}\nonumber\\
&=&\begin{pmatrix}
u^\text{T} & 0
\end{pmatrix}
\begin{pmatrix}
A_1 & A_2\\
A_3 & A_4
\end{pmatrix}
\begin{pmatrix}
a_1-u\\
a_2
\end{pmatrix}\nonumber\\
&=& u^\text{T}(A_1 a_1+A_2 a_2-A_1 u)\nonumber\\
&=& u^\text{T}(\mathbf{y}-\mathbf{y})=0.\label{orthogonality}
\end{eqnarray}

For a general $f\in\mathcal{N}_\Phi(\Omega)$, we can find a sequence $f_n\in F_\Phi$ with $f_n\rightarrow f$ in $\mathcal{N}_\Phi(\Omega)$ as $n\rightarrow\infty$. The desired result then follows from a limiting form of (\ref{orthogonality}).
\end{proof}

\begin{proof}[Proof of Lemma \ref{th:equivalence}]
For any $g\in L_2(\mathbf{R}^d)\cap C(\mathbf{R}^d)$, its native norm admits the representation
\begin{eqnarray}
\|g\|^2_{\mathcal{N}_\Phi(\mathbf{R}^d)}=(2\pi)^{-d/2}\int_{\mathbf{R}^d}\frac{|\tilde{g}(\omega)|^2}{\tilde{\Phi}(\omega)}d \omega,\label{nativerepresentation}
\end{eqnarray}
where $\tilde{g}$ and $\tilde{\Phi}$ denote the Fourier transforms of $g$ and $\Phi$ respectively. See Theorem 10.12 of \cite{wendland2005scattered}. The (fractional) Sobolev norms have a similar representation
\begin{eqnarray}
\|g\|^2_{H^s(\mathbf{R}^d)}=(2\pi)^{-d/2}\int_{\mathbf{R}^d} |\tilde{g}(\omega)|^2 (1+\|\omega\|^2)^s d \omega.\label{sobolevrepresentation}
\end{eqnarray}
See \cite{adams2003sobolev} for details.
\cite{tuo2014calibration} show that
\begin{eqnarray*}
\tilde{C}_{\upsilon,\gamma}(\omega)=2^{d/2}(4\upsilon\gamma^2)^\upsilon\frac{\Gamma(\upsilon+d/2)}{\Gamma(\nu)} (4\upsilon\gamma^2+\|\omega\|^2)^{-(\upsilon+d/2)}.
\end{eqnarray*}
Using the inequality
\begin{eqnarray*}
(1+b)\min(1,a)\leq a+b\leq (1+b)\max(1,a),
\end{eqnarray*}
for $a,b\geq 0$, we obtain
\begin{eqnarray}
\tilde{C}_{\upsilon,\gamma}(\omega)&\leq& 2^{d/2}(4\upsilon\gamma^2)^\upsilon\frac{\Gamma(\upsilon+d/2)}{\Gamma(\upsilon)} \max\left\{1,(4\upsilon\gamma^2)^{-(\upsilon+d/2)}\right\}(1+\|\omega\|^2)^{-(\upsilon+d/2)}\nonumber\\
&\leq& 2^{d/2}\frac{\Gamma(\upsilon+d/2)}{\Gamma(\upsilon)} \max\left\{(4\upsilon\gamma^2_2)^\upsilon,(4\upsilon\gamma^2_1)^{-d/2}\right\} (1+\|\omega\|^2)^{-(\upsilon+d/2)}\nonumber\\
&=:& C_1 (1+\|\omega\|^2)^{-(\upsilon+d/2)},\label{C1}
\end{eqnarray}
and
\begin{eqnarray}
\tilde{C}_{\upsilon,\gamma}(\omega)&\geq& 2^{d/2}(4\upsilon\gamma^2)^\upsilon\frac{\Gamma(\upsilon+d/2)}{\Gamma(\upsilon)} \min\left\{1,(4\upsilon\gamma^2)^{-(\upsilon+d/2)}\right\}(1+\|\omega\|^2)^{-(\upsilon+d/2)}\nonumber\\
&\geq&2^{d/2}\frac{\Gamma(\upsilon+d/2)}{\Gamma(\upsilon)} \min\left\{(4\upsilon\gamma^2_1)^\upsilon,(4\upsilon\gamma^2_2)^{-d/2}\right\} (1+\|\omega\|^2)^{-(\upsilon+d/2)}\nonumber\\
&=:& C_2 (1+\|\omega\|^2)^{-(\upsilon+d/2)},\label{C2}
\end{eqnarray}
hold for all $\omega\in\mathbf{R}^d$.

Now we apply the extension theorem of the native spaces (Theorem 10.46 of \citealp{wendland2005scattered}) to obtain a function $f^E\in\mathcal{N}_{C_{\upsilon,\gamma}}(\mathbf{R}^d)$ such that $f^E|_\Omega=f$ and $\|f\|_{\mathcal{N}_{C_{\upsilon,\gamma}}(\Omega)}=\|f^E\|_{\mathcal{N}_{C_{\upsilon,\gamma}}(\mathbf{R}^d)}$ for each $\gamma\in[\gamma_1,\gamma_2]$. We use (\ref{nativerepresentation})-(\ref{C1}) to obtain
\begin{eqnarray}
\|f\|^2_{\mathcal{N}_{C_{\upsilon,\gamma}}(\Omega)}&=&\|f^E\|^2_{\mathcal{N}_{C_{\upsilon,\gamma}}(\mathbf{R}^d)}
=(2\pi)^{-d/2}\int_{\mathbf{R}^d}\frac{|\tilde{f}^E(\omega)|^2}{\tilde{C}_{\upsilon,\gamma}(\omega)} d \omega\nonumber\\
&\geq&C_1^{-1}(2\pi)^{-d/2}\int_{\mathbf{R}^d}|\tilde{f}^E(\omega)|^2(1+\|\omega\|^2)^{\upsilon+d/2} d \omega\nonumber\\
&=&C_1^{-1}\|f^E\|^2_{H^{\upsilon+d/2}(\mathbf{R}^d)}\geq C_1^{-1}\|f\|^2_{H^{\upsilon+d/2}(\Omega)},\label{normineq1}
\end{eqnarray}
where the last inequality follows from the fact that $f^E|_\Omega=f$. On the other hand, because $\Omega$ is convex, $f$ has an extension $f_E\in H^{\upsilon+d/2}(\mathbf{R}^d)$ satisfying $\|f_E\|_{H^k(\mathbf{R}^d)}\leq c\|f\|_{H^k(\Omega)}$ for some constant $c$ independent of $f$. Then we use (\ref{nativerepresentation}), (\ref{sobolevrepresentation}) and (\ref{C2}) to obtain
\begin{eqnarray*}
\|f_E\|^2_{H^k(\Omega)}&\geq& c^{-2}\|f\|^2_{H^k(\Omega)}\nonumber\\
&=& c^{-2}(2\pi)^{-d/2} \int_{\mathbf{R}^d}|\tilde{f}^E(\omega)|^2(1+\|\omega\|^2)^{\upsilon+d/2} d \omega\nonumber\\
&\geq& c^{-2} C_2(2\pi)^{-d/2}\int_{\mathbf{R}^d}\frac{|\tilde{f}^E(\omega)|^2}{\tilde{C}_{\upsilon,\gamma}(\omega)} d \omega\nonumber\\
&=&c^{-2}C_2\|f_E\|^2_{\mathcal{N}_{C_{\upsilon,\gamma}}(\mathbf{R}^d)}\geq c^{-2}C_2\|f\|^2_{\mathcal{N}_{C_{\upsilon,\gamma}}(\Omega)},\label{normineq2}
\end{eqnarray*}
where the last inequality follows from the restriction theorem of the native space, which states that the restriction $f=f_E|_\Omega$ is contained in $\mathcal{N}_{C_{\upsilon,\gamma}}(\Omega)$ with a norm that is less than or equal to the norm $\|f_E\|_{\mathcal{N}_{C_{\upsilon,\gamma}}(\mathbf{R}^d)}$. See Theorem 10.47 of \cite{wendland2005scattered}. The desired result is proved by combining (\ref{normineq1}) and (\ref{normineq2}).
\end{proof}

\begin{proof}[Proof of Lemma \ref{th:representer}]
For $f\in\mathcal{N}_\Phi(\Omega)$, define
\begin{eqnarray*}
M(f)=L(f(x_1),\ldots,f(x_n))+\|f\|^2_{\mathcal{N}_\Phi(\Omega)}.
\end{eqnarray*}
Now consider $s_{\hat{f},X}$, i.e., the interpolant of $\hat{f}$ over $X=\{x_1,\ldots,x_n\}$ using the kernel function $\Phi$. Because $\hat{f}(x_i)=s_{\hat{f},X}(x_i)$ for $i=1,\ldots,n$, we have
\begin{eqnarray}
L(f(x_1),\ldots,f(x_n))=L(s_{\hat{f},X}(x_1),\ldots,s_{\hat{f},X}(x_n)). \label{representer.e1}
\end{eqnarray}
In addition, it is easily seen from Lemma \ref{lemma:1}, (\ref{pathagorean}) and (\ref{optimality}) that
\begin{eqnarray}
\|s_{\hat{f},X}\|^2_{\mathcal{N}_\Phi(\Omega)}\leq \|\hat{f}\|^2_{\mathcal{N}_\Phi(\Omega)},\label{representer.e2}
\end{eqnarray}
and the equality holds if and only if $s_{\hat{f},X}=\hat{f}$. By combining (\ref{representer.e1}) and (\ref{representer.e2}) we obtain
\begin{eqnarray}
M(s_{\hat{f},X})\leq M(\hat{f}).\label{representer.e3}
\end{eqnarray}
Because $\hat{f}$ minimizes $M(f)$, the reverse of (\ref{representer.e3}) also holds. Hence we deduce $s_{\hat{f},X}=\hat{f}$, which proves the theorem according to the definition of the interpolant.
\end{proof}

\begin{proof}[Proof of Theorem \ref{th:leastsquares}]
We first rewrite the minimization problem (\ref{epsilon}) as the following iterated form
\begin{eqnarray}\label{iterated}
&&\min_{\substack{\theta\in\Theta\\f\in\mathcal{N}_{C_{\upsilon,\gamma}}(\Omega)}}\sum_{i=1}^n (y_i^p-y^s(x_i,\theta)-f(x_i))^2+\frac{\sigma^2}{\tau^2} \|f\|^2_{\mathcal{N}_{C_{\upsilon,\gamma}}(\Omega)}\nonumber\\
&=&\min_{\theta\in\Theta}\min_{f\in\mathcal{N}_{C_{\upsilon,\gamma}}(\Omega)} \sum_{i=1}^n (y_i^p-y^s(x_i,\theta)-f(x_i))^2+\frac{\sigma^2}{\tau^2} \|f\|^2_{\mathcal{N}_{C_{\upsilon,\gamma}}(\Omega)}
\end{eqnarray}
Now we apply Lemma \ref{th:representer} to the inner minimization problem in (\ref{iterated}) and obtain the following representation for $\hat{\Delta}$:
\begin{eqnarray*}
\hat{\Delta}=\sum_{i=1}^n \alpha_i C_{\upsilon,\gamma}(x_i,\cdot),
\end{eqnarray*}
with an undetermined vector of coefficients $\alpha=(\alpha_1,\ldots,\alpha_n)^\text{T}$. Using the definition $\Sigma_{\gamma}=(C_{\upsilon,\gamma}(x_i,x_j))_{i j}$, clearly we have the matrix representation
\begin{eqnarray}
\hat{\Delta}(\mathbf{x})=\Sigma_\gamma \alpha .\label{changeofvariable}
\end{eqnarray}
Now using (\ref{innerproduct}) we have
\begin{eqnarray*}
\|\hat{\Delta}\|_{\mathcal{N}_{C_{\upsilon,\gamma}}(\Omega)}^2=\left\langle\sum_{i=1}^n \alpha_i C_{\upsilon,\gamma}(x_i,\cdot),\sum_{i=1}^n \alpha_i C_{\upsilon,\gamma}(x_i,\cdot)\right\rangle_{\mathcal{N}_{C_{\upsilon,\gamma}}(\Omega)}=\alpha^\text{T}\Sigma_{\gamma}\alpha.
\end{eqnarray*}
The minimization problem (\ref{epsilon}) then reduces to
\begin{eqnarray*}
\operatorname*{argmin}_{\substack{\theta\in\Theta\\ \alpha\in\mathbf{R}^n}} \|\mathbf{y}^p-y^s(\mathbf{x},\theta)-\alpha\Sigma_{\gamma}\|^2+\frac{\sigma^2}{\tau^2} \alpha^\text{T}\Sigma_{\gamma}\alpha.
\end{eqnarray*}
Applying a change-of-variable argument using (\ref{changeofvariable}) we obtain the following optimization formula
\begin{eqnarray*}
\operatorname*{argmin}_{\substack{\theta\in\Theta\\ \Delta(\mathbf{x})\in\mathbf{R}^n}} \|\mathbf{y}^p-y^s(\mathbf{x},\theta)-\Delta(\mathbf{x})\|^2+\frac{\sigma^2}{\tau^2} \Delta(\mathbf{x})^\text{T}\Sigma_{\gamma}^{-1}\Delta(\mathbf{x}).
\end{eqnarray*}
Elementary calculations show its equivalence to the definition of $(\hat{\theta}_{KO},\hat{\delta}(\mathbf{x}))$.
\end{proof}

\section*{Acknowledgements}
Tuo's work is supported by the National Center for Mathematics and Interdisciplinary Sciences in CAS
and NSFC grant 11501551, 11271355 and 11671386.
Wu's work is supported by NSF grant DMS 1564438.
The authors are grateful to the Associate Editor and referees for very helpful comments.

\bibliographystyle{chicago}
\bibliography{calibration}

\end{document}